\documentclass[a4paper,12pt]{article}

\usepackage[utf8]{inputenc}
\usepackage[T1]{fontenc}
\usepackage[a4paper, margin=2.5cm]{geometry}

\usepackage{amsmath, amsthm, amssymb, enumerate}
\usepackage{graphicx}
\usepackage{enumerate}
\usepackage{authblk}
\usepackage[textwidth=2.1cm, color=orange!70]{todonotes}
\usepackage[absolute]{textpos}
\usepackage{comment}
\usepackage{dsfont}
\usepackage{needspace}
\usepackage{mathtools}

\usepackage{thmtools}

\usepackage{tikz}
\usetikzlibrary{decorations.pathreplacing}
\tikzstyle{small node} = [draw, circle, fill = black, minimum size = 3pt, inner sep = 0pt]
\tikzstyle{black node} = [draw, circle, fill = black, minimum size = 5pt, inner sep = 0pt]
\tikzstyle{white node} = [draw, circle, fill = white, minimum size = 5pt, inner sep = 0pt]
\tikzstyle{normal} = [draw=none, fill = none, rectangle, minimum size =0]

\usepackage{ifthen}
\def\empty{}
\pgfkeys{utils/.cd,
	set if not empty/.code n args={3}{%
		\def\arg{#2}%
		\ifx\arg\empty%
		\pgfkeys{#1={#3}}%
		\else%
		\pgfkeys{#1={#2}}%
		\fi%
	},
	set if labelpos not empty/.code n args={2}{%
		\def\arg{#1}%
		\ifx\arg\empty%
		\else%
		\def\argo{#2}%
		\ifx\argo\empty%
		\pgfkeys{/tikz/label={#1}}%
		\else%
		\pgfkeys{/tikz/label={#2}:{#1}}%
		\fi%
		\fi%
	}
}
\tikzset{
	nodes/.style n args={4}{
		draw ,circle,outer sep=0.5mm,
		/utils/set if not empty={/tikz/fill}{#1}{black},
		/utils/set if not empty={/tikz/minimum size}{#4}{5}
	}
}

\usepackage{algorithm}
\usepackage[noend]{algpseudocode}

\usepackage{hyperref}
\usepackage{cleveref}

\usepackage[absolute]{textpos}

\newtheorem{environment}{Environment}[section]

\newtheorem{lemma}[environment]{Lemma}
\crefname{lemma}{lemma}{lemmas}
\crefformat{lemma}{#2Lemma~#1#3}
\Crefformat{lemma}{#2Lemma~#1#3}

\crefname{question}{question}{questions}
\crefformat{question}{#2Question~#1#3}
\Crefformat{question}{#2Question~#1#3}

\crefname{corollary}{corollary}{corollaries}
\crefformat{corollary}{#2Corollary~#1#3}
\Crefformat{corollary}{#2Corollary~#1#3}

\newtheorem{theorem}{Theorem}
\crefname{theorem}{theorem}{theorems}
\crefformat{theorem}{#2Theorem~#1#3}
\Crefformat{theorem}{#2Theorem~#1#3}

\crefname{proposition}{proposition}{Propositions}
\crefformat{proposition}{#2Proposition~#1#3}
\Crefformat{proposition}{#2Proposition~#1#3}

\newtheorem{Conjecture}{Conjecture}
\crefname{Conjecture}{Conjecture}{Conjectures}
\crefformat{Conjecture}{#2Conjecture~#1#3}
\Crefformat{Conjecture}{#2Conjecture~#1#3}

\crefname{example}{example}{examples}
\crefformat{example}{#2Example~#1#3}
\Crefformat{example}{#2Example~#1#3}

\crefname{remark}{remark}{remarks}
\crefformat{remark}{#2Remark~#1#3}
\Crefformat{remark}{#2Remark~#1#3}

\crefname{definition}{definition}{definitions}
\crefformat{definition}{#2Definition~#1#3}
\Crefformat{definition}{#2Definition~#1#3}

\crefname{figure}{figure}{figures}
\crefformat{figure}{#2Figure~#1#3}
\Crefformat{figure}{#2Figure~#1#3}

\crefname{chapter}{chapter}{chapters}
\crefformat{chapter}{#2Chapter~#1#3}
\Crefformat{chapter}{#2Chapter~#1#3}

\crefname{section}{section}{sections}
\crefformat{section}{#2Section~#1#3}
\Crefformat{section}{#2Section~#1#3}

\crefname{algorithm}{algorithm}{algorithms}
\crefformat{algorithm}{#2Algorithm~#1#3}
\Crefformat{algorithm}{#2Algorithm~#1#3}

\crefname{notation}{notation}{notations}
\crefformat{notation}{#2Notation~#1#3}
\Crefformat{notation}{#2Notation~#1#3}

\newtheorem{claim}{Claim}
\crefname{claim}{claim}{claims}
\crefformat{claim}{#2Claim~#1#3}
\Crefformat{claim}{#2Claim~#1#3}

\crefname{enumi}{condition}{conditions}
\crefformat{enumi}{#2Condition~#1#3}
\Crefformat{enumi}{#2Condition~#1#3}

\crefname{definitionx}{definition}{definitions}
\crefformat{definitionx}{#2Definition~#1#3}
\Crefformat{definitionx}{#2Definition~#1#3}

\def\\cp{\mathcal{C'}} 

\renewcommand{\geq}{\geqslant}
\renewcommand{\leq}{\leqslant}
\renewcommand{\ge}{\geqslant}
\renewcommand{\le}{\leqslant}

\def\cqedsymbol{\ifmmode$\lrcorner$\else{\unskip\nobreak\hfil
\penalty50\hskip1em\null\nobreak\hfil$\lrcorner$
\parfillskip=0pt\finalhyphendemerits=0\endgraf}\fi}

\def\lqedsymbol{\ifmmode$\lrcorner$\else{\unskip\nobreak\hfil
\penalty50\hskip1em\null\nobreak\hfil$\rule{1ex}{1ex}$
\parfillskip=0pt\finalhyphendemerits=0\endgraf}\fi} 

\newcommand{\cqed}{\renewcommand{\qed}{\cqedsymbol}}

\DeclareMathOperator{\fvs}{fvs}
\DeclareMathOperator{\cp}{cp}
\DeclareMathOperator{\fp}{fp}

\title{Jones' Conjecture in subcubic graphs%
\thanks{This research is a part of a project that have received funding from the European Research Council (ERC) under the European Union's Horizon 2020 research and innovation programme Grant Agreement 714704.}}

\author[1]{Marthe Bonamy}
\author[2]{Fran\c{c}ois Dross}
\author[2]{Tom\'a\v{s} Masa\v{r}\'ik}
\author[2]{Wojciech Nadara}
\author[2]{Marcin Pilipczuk}
\author[2]{Micha\l \ Pilipczuk}
\affil[1]{CNRS, LaBRI, Université de Bordeaux, France.}
\affil[2]{MIMUW, Warsaw, Poland.}
\date{\today}

\begin{document}

\maketitle

\begin{abstract}
We confirm Jones' Conjecture for subcubic graphs. Namely, if a subcubic planar graph does not contain $k+1$ vertex-disjoint cycles, then it suffices to delete $2k$ vertices to obtain a forest.
\end{abstract}

\begin{textblock}{20}(0, 13.4)
\includegraphics[width=40px]{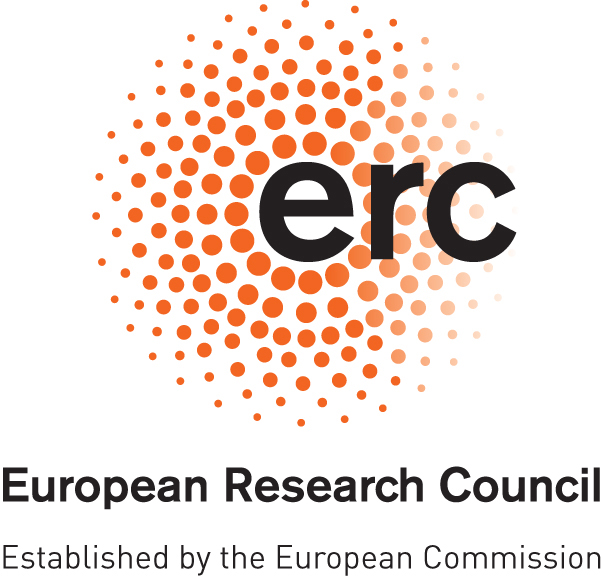}%
\end{textblock}
\begin{textblock}{20}(0, 14.3)
\includegraphics[width=40px]{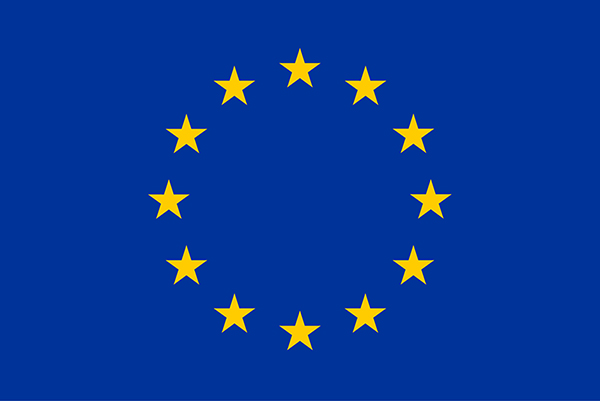}%
\end{textblock}


We investigate the connection between the maximum number of vertex-disjoint cycles in a graph and the minimum number of vertices whose deletion results in a cycle-free graph, i.e. a forest. A \emph{cycle packing} of a (multi)graph $G$ is a set of vertex disjoint cycles that appear in $G$ as subgraphs. We denote the maximum size of a cycle packing of $G$ by $\cp(G)$. A \emph{feedback vertex set} of a (multi)graph $G$ is a set $S$ of vertices such that $G - S$ is a forest. We denote an arbitrary minimum feedback vertex set of $G$ as $FVS(G)$ and denote its size by $\fvs(G)$.

Erd\H{o}s and Pósa~\cite{erdos_posa_1965} showed that there is a constant $c$ such that for any graph $G$, $\fvs(G) \le c \cdot\cp(G) \log \cp(G)$, and that this upper-bound is tight for some graphs. This seminal result led to rich developments. Two main directions: determining which structures can replace ``cycles'' there (free to increase the bounding function), and improving the bounding functions. Most notably, in the first line of research, Cames van Banteburg et al.~\cite{van2019tight} proved that ``cycle'' can directly be replaced with ``minor of a given planar graph'', which was earlier showed by Robertson and Seymour~\cite{robertson1986graph} with a worse bounding function.

We focus on the second line of research, and are interested in the best possible bound in the original theorem, when restricted to the specific case of planar graphs. A drastic improvement is then possible, and we were motivated by the following elusive conjecture.

\begin{Conjecture}[``Jones' Conjecture'', Kloks, Lee and Liu~\cite{kloks2002new}]\label{conj:jones}
Every planar graph $G$ satisfies $\fvs(G) \leq 2 \cdot \cp(G)$.\footnote{\url{http://www.openproblemgarden.org/op/jones_conjecture}}
\end{Conjecture}

Note that Conjecture~\ref{conj:jones} is tight for wheels or for the dodecahedron.
Currently, the best known bound is that every planar graph $G$ satisfies $\fvs(G) \leq 3 \cdot \cp(G)$, as proved independently by Chappel et al.~\cite{chappell2014cycle}, Chen et al.~\cite{chen2012feedback}, and Ma et al.~\cite{ma2013approximate}.

In his PhD Thesis, Munaro~\cite{munaro2016quelques} considered the case of subcubic graphs and made significant progress. Here we complete the case, and prove that Jones' Conjecture holds for subcubic graphs.

\begin{theorem}\label{th:jonessubcubic}
Every subcubic planar multigraph $G$ satisfies $\fvs(G) \leq 2 \cdot \cp(G)$.
\end{theorem}

\section{Proof of Theorem~\ref{th:jonessubcubic}}

\subsection{Notation}
A multigraph is \emph{simple} if it has no loops or multi-edges. In this case, we simply refer to it as a graph.

Let $G = (V,E)$ be a (multi)graph. For $W \subseteq V$, we denote by $G[W]$ the subgraph of $G$ induced by $W$, and by $G - W$ the subgraph of $G$ induced by $V \setminus W$. If $W = \{v\}$, then we denote $G - v = G - W$. For 
$F \subseteq E$, we denote $G - F = (V, E \setminus F)$. If $F = \{e\}$, then we denote $G - e = G - F$. For $W \cap V = \emptyset$, $G + W$ is the disjoint union of $G$ and a set $W$ of isolated vertices. If $W = \{v\}$, then we denote $G + v = G + W$. For $F$ a set of pairs of edges with $F \cap E = \emptyset$, we denote $G + F = (V, E \cup F)$. If $F = \{e\}$, then we denote $G + e = G + F$. Graph is called \emph{cubic} if all of its degrees are exactly $3$. Graph is called \emph{subcubic} if its degrees are not bigger than $3$.

A (multi)graph is \emph{$k$-connected} if 
the removal of any $k-1$ vertices leaves the graph connected. A (multi)graph is \emph{$k$-edge-connected} if the removal of at most $k-1$ edges leaves the graph connected. Note that a subcubic (multi)graph with at least $k+1$ vertices is $k$-connected if and only if it is $k$-edge-connected.

A (multi)graph is \emph{essentially $4$-edge-connected} if the removal of at most three edges does not yield two components with at least two vertices each. A (multi)graph is \emph{cyclically $4$-edge-connected} if the removal of at most three edges does not yield two components that both contain a cycle. For a cubic (multi)graph, these last two notions are equivalent.

\subsection{Proof}

We proceed by contradiction. Let $G$ be a counter-example to Theorem~\ref{th:jonessubcubic} with the fewest vertices. We use the following very convenient theorem from Munaro~\cite{munaro2016quelques}:

\begin{theorem}[Theorem 3.4.10 in~\cite{munaro2016quelques}]\label{munaro}
If $G$ is simple, 
then 
it is not cyclically 4-edge-connected.
\end{theorem}

To obtain a contradiction, we argue that $G$ is in fact, a simple graph that is essentially $4$-edge-connected, as follows.

\begin{lemma}\label{lem:3connected}



The multigraph $G$ is an essentially $4$-edge-connected simple graph.

\end{lemma}



\begin{proof}

While Claims~\ref{cl:3reg},~\ref{2co},~\ref{3co} are known and easy properties of a minimum counter-example to Jones' conjecture on subcubic graphs (see e.g.~\cite{munaro2016quelques}), we include their proofs because we believe they may constitute a useful warm-up. The uninterested reader may skip them guilt-free.

\begin{claim}\label{cl:3reg}
The multigraph $G$ is $3$-regular.
\end{claim}
\begin{proof}

Suppose $G$ has a vertex $v$ with degree at most $1$. Then $G - v$ satisfies Jones' Conjecture by minimality of $G$. As no cycle of $G$ contains $v$, we have $\fvs(G - v) = \fvs(G)$ and $\cp(G - v) = \cp(G)$, therefore $G$ also satisfies Jones' Conjecture, a contradiction. 

Suppose $G$ has a vertex $v$ with degree $2$, and let $u$ and $w$ be the two neighbors of $v$. Then $G' = G - v + uw$ satisfies Jones' Conjecture, so $\fvs(G') \leq 2 \cdot \cp(G')$. The cycles of $G$ are in bijection with the cycles of $G'$, by exchanging the edges $uv$ and $vw$ and an edge $uw$ when appropriate. Hence $\cp(G) = \cp(G')$. Moreover, if $S$ is a feedback vertex set of $G$ that does not contain $v$, then $S$ is a feedback vertex set of $G'$, and if $S$ is a feedback vertex set of $G$ that contains $v$, then $(S \setminus \{v\}) \cup \{u\}$ is a feedback vertex set of $G'$. Thus $\fvs(G) \le \fvs(G') \leq 2 \cdot \cp(G') = \cp(G)$, and $G$ satisfies Jones' Conjecture, a contradiction. Hence $G$ is cubic.
\cqed\end{proof}

\begin{claim}\label{2co}
The multigraph $G$ is $2$-connected.
\end{claim}

\begin{proof}
Suppose that $G$ is not $2$-connected. As $G$ is cubic, that means that $G$ is not $2$-edge-connected. Let $e$ be a separating edge of $G$. Both components $G_1$ and $G_2$ of $G - e$ verify Jones' Conjecture by minimality of $G$. Since $e$ is separating, it is not in any cycle of $G$. The union of any feedback vertex set of $G_1$ and any feedback vertex set of $G_2$ is a feedback vertex set of $G$, so $\fvs(G) \le \fvs(G_1) + \fvs(G_2)$. The union of any cycle packing of $G_1$ and any cycle packing of $G_2$ is a cycle packing of $G$, so $\cp(G) \ge \cp(G_1) + \cp(G_2)$. Therefore $G$ satisfies Jones' Conjecture, a contradiction. 
\cqed\end{proof}
In particular since $G$ is cubic, Claim~\ref{2co} implies that $G$ is a simple graph.

\begin{claim}\label{3co}
The graph $G$ is $3$-connected.
\end{claim}

\begin{proof} Assume that it is not $3$-connected, and thus not $3$-edge-connected. Let $u_1u_2$ and $v_1v_2$ be a $2$-edge-cut, where $u_1$ and $v_1$ are in the same connected component of \mbox{$G-\{u_1u_2,v_1v_2\}$}, which we denote $G_1$. Let $G_2$ be the other connected component of \mbox{$G-\{u_1u_2,v_1v_2\}$}. We write $G'_1=G_1+ \{u_1v_1\}$ and $G'_2=G_2+\{u_2v_2\}$. Note that this may lead to a double edge. See Figure~\ref{fig3co} for an illustration. By minimality of $G$, we know that $G_1$, $G_2$, $G_1'$ and $G_2'$ all satisfy Jones' Conjecture.

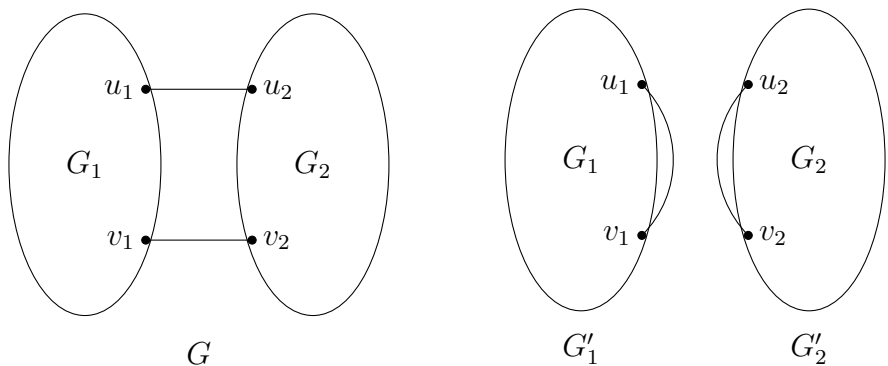
\begin{figure}[h]
\begin{center}
\begin{tikzpicture}
    \draw (0,0) ellipse (1cm and 2cm);
    \draw (0,0) node {$G_1$};
    \coordinate (u1) at (0.8,1);
    \coordinate (v1) at (0.8,-1);

    \draw[left] (u1) node {$u_1$};
    \draw[left] (v1) node {$v_1$};

    \draw [fill=black](u1) circle (1.5pt) ;
    \draw [fill=black](v1) circle (1.5pt) ;
    
    \draw (3,0) ellipse (1cm and 2cm);
    \draw (3,0) node {$G_2$};
    \coordinate (u2) at (2.2,1);
    \coordinate (v2) at (2.2,-1);

    \draw[right] (u2) node {$u_2$};
    \draw[right] (v2) node {$v_2$};

    \draw [fill=black](u2) circle (1.5pt) ;
    \draw [fill=black](v2) circle (1.5pt) ;
    
    \draw (v1) -- (v2);
    \draw (u1) -- (u2);
    \draw (1.5,-2.5) node {$G$};
    \end{tikzpicture}~~~~~~~~~~
\begin{tikzpicture}
    \draw (0,0) ellipse (1cm and 2cm);
    \draw (0,0) node {$G_1$};
    \coordinate (u1) at (0.8,1);
    \coordinate (v1) at (0.8,-1);

    \draw[left] (u1) node {$u_1$};
    \draw[left] (v1) node {$v_1$};

    \draw [fill=black](u1) circle (1.5pt) ;
    \draw [fill=black](v1) circle (1.5pt) ;
    \draw (0,-2.5) node {$G_1'$};
    
    \draw (3,0) ellipse (1cm and 2cm);
    \draw (3,0) node {$G_2$};
    \coordinate (u2) at (2.2,1);
    \coordinate (v2) at (2.2,-1);

    \draw[right] (u2) node {$u_2$};
    \draw[right] (v2) node {$v_2$};

    \draw [fill=black](u2) circle (1.5pt) ;
    \draw [fill=black](v2) circle (1.5pt) ;
    
    \draw (v1) to [out=45, in=-45] (u1);
    \draw (v2) to [out=134, in=-135] (u2);
    \draw (3,-2.5) node {$G_2'$};
    \end{tikzpicture}
    
\caption{The graphs $G$, $G_1'$, and $G_2'$ in Claim~\ref{3co}. \label{fig3co}}
\end{center}
\end{figure}

Note that since $G'_1=G_1+ \{u_1v_1\}$, we have $\cp(G_1) \le \cp(G'_1) \le \cp(G_1)+1$. We first argue that $\cp(G'_1)=\cp(G_1)+1$. Assume for a contradiction that $\cp(G'_1)=\cp(G_1)$. Note that for any feedback vertex set $S_1$ of $G'_1$, either $u_1 \in S_1$ or $v_1 \in S_1$ or $u_1$ and $v_1$ are in distinct components of $G_1 - S_1$, so $\fvs(G)\leq \fvs(G'_1)+\fvs(G_2)$. Thus, $\fvs(G)\leq \fvs(G'_1)+\fvs(G_2)\leq 2 \cp(G'_1)+2 \cp(G_2)=2 \cp(G_1)+2 \cp(G_2)\leq 2 \cp(G)$, a~contradiction. 

By symmetry, we have $\cp(G'_2)=\cp(G_2)+1$. Therefore, every cycle packing of $G'_1$ contains the edge $u_1v_1$ and every cycle packing of $G'_2$ contains the edge $u_2v_2$. We can thus combine a cycle packing of $G'_1$ and a cycle packing of $G'_2$ by making a single cycle out of those two cycles. So $\cp(G)=\cp(G_1)+\cp(G_2)+1$. However, if $S_1$ is a feedback vertex set of $G_1$ and $S_2$ is a feedback vertex set of $G_2$, then $S_1 \cup S_2 \cup \{u_1\}$ is a feedback vertex set of $G$. Therefore $\fvs(G)\leq \fvs(G_1)+\fvs(G_2)+1 \leq 2\cp(G_1)+2\cp(G_2)+1 < 2 \cp(G)$, a~contradiction. Therefore $G$ is $3$-connected.
\cqed\end{proof}


\begin{claim}
The graph $G$ is essentially $4$-edge-connected.
\end{claim}

\begin{proof}
Assume that $G$ is not essentially $4$-edge-connected, and thus not cyclically $4$-edge-connected.
Consider a non-trivial $3$-edge-cut $\{e_A,e_B,e_C\}$. Let $G_1$ and $G_2$ be the two components of $G \setminus \{e_A,e_B,e_C\}$. For $i \in \{1,2\}$, we define $G_i^{ABC}$ as the graph obtained from $G$ by contracting $G_{3-i}$ into a single vertex $x$. See Figure~\ref{figend} for an illustration.
We define $G_i^{AB}$ (resp. $G_i^{AC}$, $G_i^{BC}$) as the graph obtained from $G_i$ by connecting with an edge vertices from $G_i$ incident to $e_A$ and $e_B$ (resp. to $e_A$ and $e_C$ for $G_i^{AC}$ or to $e_B$ and $e_C$ for $G_i^{BC}$). Again, this may lead to a double edge.
Note that for both values of $i$, all of $G_i$, $G_i^{AB}$, $G_i^{AC}$, $G_i^{BC}$ and $G_i^{ABC}$ have fewer vertices than $G$, and thus satisfy Jones' Conjecture.

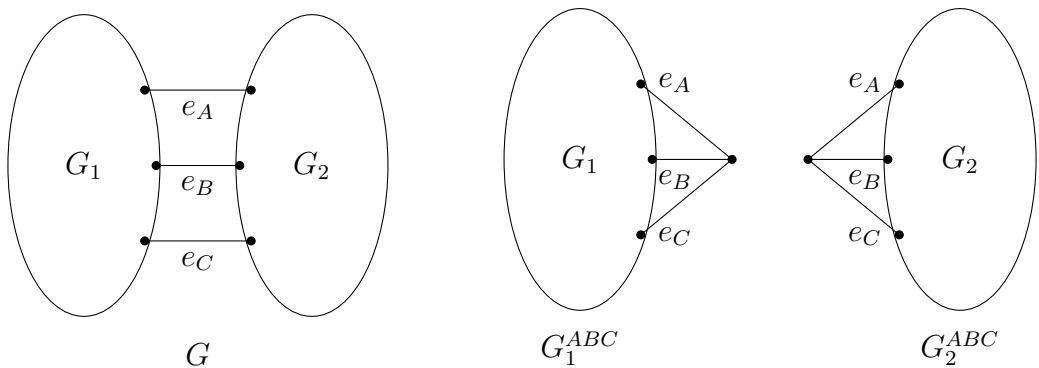
\begin{figure}[h]
\begin{center}
\begin{tikzpicture}
    \draw (0,0) ellipse (1cm and 2cm);
    \draw (0,0) node {$G_1$};
    \coordinate (u1) at (0.8,1);
    \coordinate (w1) at (0.95,0);
    \coordinate (v1) at (0.8,-1);

    \draw [fill=black](u1) circle (1.5pt) ;
    \draw [fill=black](v1) circle (1.5pt) ;
    \draw [fill=black](w1) circle (1.5pt) ;
    
    \draw (3,0) ellipse (1cm and 2cm);
    \draw (3,0) node {$G_2$};
    \coordinate (u2) at (2.2,1);
    \coordinate (w2) at (2.05,0);
    \coordinate (v2) at (2.2,-1);
      
    \draw [fill=black](u2) circle (1.5pt) ;
    \draw [fill=black](v2) circle (1.5pt) ;
    \draw [fill=black](w2) circle (1.5pt) ;
    
    \draw (v1) -- (v2);
    \draw (u1) -- (u2);
    \draw (w1) -- (w2);
    \draw (1.5,-2.5) node {$G$};
    \draw[below] (1.5,1) node {$e_A$};
    \draw[below] (1.5,0) node {$e_B$};
    \draw[below] (1.5,-1) node {$e_C$};
    \end{tikzpicture}~~~~~~~~~~
\begin{tikzpicture}
    \draw (0,0) ellipse (1cm and 2cm);
    \draw (0,0) node {$G_1$};
    \coordinate (u1) at (0.8,1);
    \coordinate (w1) at (0.95,0);
    \coordinate (v1) at (0.8,-1);
    \coordinate (x1) at (2,0);
    
    \draw (v1) -- (x1);
    \draw (u1) -- (x1);
    \draw (w1) -- (x1);

    \draw [fill=black](u1) circle (1.5pt) ;
    \draw [fill=black](v1) circle (1.5pt) ;
    \draw [fill=black](w1) circle (1.5pt) ;
    \draw [fill=black](x1) circle (1.5pt) ;
    \draw (0,-2.5) node {$G_1^{ABC}$};
    \draw[above] (1.25,0.75) node {$e_A$};
    \draw[below] (1.25,0) node {$e_B$};
    \draw[below] (1.25,-0.75) node {$e_C$};
    
    \draw (5,0) ellipse (1cm and 2cm);
    \draw (5,0) node {$G_2$};
    \coordinate (u2) at (4.2,1);
    \coordinate (w2) at (4.05,0);
    \coordinate (v2) at (4.2,-1);
    \coordinate (x2) at (3,0);
    
    \draw (v2) -- (x2);
    \draw (u2) -- (x2);
    \draw (w2) -- (x2);
      
    \draw [fill=black](u2) circle (1.5pt) ;
    \draw [fill=black](v2) circle (1.5pt) ;
    \draw [fill=black](w2) circle (1.5pt) ;
    \draw [fill=black](x2) circle (1.5pt) ;
    \draw[above] (3.75,0.75) node {$e_A$};
    \draw[below] (3.75,0) node {$e_B$};
    \draw[below] (3.75,-0.75) node {$e_C$};
    
    \draw (5,-2.5) node {$G_2^{ABC}$};
    \end{tikzpicture}
    
\caption{The graphs $G$, $G_1^{ABC}$, and $G_2^{ABC}$. \label{figend}}
\end{center}
\end{figure}

First note that for both values of $i$, $\fvs(G)\leq \fvs(G_i^{ABC})+\fvs(G_{3-i})$. In order to prove that let us assume without loss of generality that $i=1$. Then remove $FVS(G_2)$ from $G$. What remains from $G_2$ after deleting $FVS(G_2)$ is a forest that could hypothetically create connections between vertices from $G_1$ incident to $e_A, e_B, e_C$.
However if we are given any tree $T$ and its three vertices $u, v, w \in V(T)$ then $P_{uv} \cap P_{vw} \cap P_{wu} \neq \emptyset$ (in fact it is always a single vertex), where $P_{uv}, P_{vw}, P_{wu}$ are sets of vertices on unique paths between corresponding vertices, so it is possible to break the connections between all three pairs of these vertices by removing a single vertex of $T$.
Because of that we see that $\fvs(G - FVS(G_2)) \le \fvs(G_1^{ABC})$ what leads to $\fvs(G) \leq \fvs(G_1^{ABC})+\fvs(G_{2})$.
Therefore we see that $\fvs(G)\leq \fvs(G_i^{ABC})+\fvs(G_{3-i}) \leq \fvs(G_1)+\fvs(G_2)+1 \le 2\cp(G_1)+2\cp(G_2)+1$. We also have $\cp(G)\geq \cp(G_1)+\cp(G_2)$, yet $\fvs(G)> 2 \cp(G)$. 

It follows that for both values of $i$:

\begin{equation}\label{eq:fvsABC}
    \fvs(G_i^{ABC})=\fvs(G_i)+1
\end{equation}
\begin{equation}\label{eq:fvscp}
    \fvs(G_i)=2 \cp(G_i)
\end{equation}
And that:
\begin{equation}\label{eq:fvsG}
    \fvs(G)=\fvs(G_1)+\fvs(G_2)+1
\end{equation}
\begin{equation}\label{eq:cpG}
    \cp(G)=\cp(G_1)+\cp(G_2)
\end{equation}

We are now ready for a closer analysis.
\begin{enumerate}[(i)]
    \item\label{notthree0} For any $i$ and for any $x \in \{A,B,C\}$, we have $$\fvs(G) \leq \fvs(G_i^{ABC-x})+\max_{y \neq x} \fvs(G_{3-i}^{ABC-y}).$$ Indeed, take without loss of generality $i=1$ and $x=C$. Let us consider a minimum feedback vertex set $S$ of $G_1^{AB}$.
    Note that in $G_1\setminus S$, there is no path between the vertex incident to $e_A$ and the vertex incident to $e_B$ or at least one of them is in $S$. As a consequence, either vertex incident to $e_C$ is in $S$ or one of them, say the vertex incident to $e_A$, is either in $S$ or is not in the same component as the vertex incident to $e_C$. For any minimum feedback vertex set $S'$ of $G_2^{BC}$, we observe that $S\cup S'$ is a feedback vertex set of $G$, hence the conclusion. In particular, by combining with (\ref{eq:fvsG}), if $\fvs(G_i^{ABC-x})=\fvs(G_i)$ then $\fvs(G_{3-i}^{ABC-y})=\fvs(G_{3-i})+1$ for some $y \neq x$.
\item\label{fvscp} For every $i \in \{1,2\}$ and for every $x \in \{A,B,C\}$, if $\fvs(G_i^{ABC-x})=\fvs(G_i)+1$, then $\cp(G_i^{ABC-x})=\cp(G_i)+1$. That follows from (\ref{eq:fvscp}), since $G_i^{ABC-x}$ satisfies Jones' Conjecture.
\item\label{notwocp1} For every $x \in \{A,B,C\}$, we have either  $\cp(G_1^{ABC-x})=\cp(G_1)$ or $\cp(G_2^{ABC-x})=\cp(G_2)$. Indeed, suppose not. Then both $\cp(G_1^{ABC-x})=\cp(G_1)+1$ and $\cp(G_2^{ABC-x})=\cp(G_2)+1$, for say $x = C$. Then for $i \in \{1,2\}$, in every cycle packing of $G_i^{AB}$ there is a cycle containing $e_A$ and $e_B$. By taking a cycle packing of $G_1^{AB}$ and a cycle packing of $G_2^{AB}$, we obtain a cycle packing of $G$ (combining two cycles into one). So
 $\cp(G) \ge \cp(G_1)+\cp(G_2)+1$, a contradiction with (\ref{eq:cpG}).
\end{enumerate}

It follows from~(\ref{eq:fvsABC}) and~(\ref{eq:fvscp}) that $\cp(G_i^{ABC})=\cp(G_i)+1$ for both values of $i$. Note that a maximum cycle packing of $G_i^{ABC}$ uses two edges out of $\{e_A,e_B,e_C\}$. It follows that for some $z_i \in \{A,B,C\}$, we have $\cp(G_i^{ABC-z_i})=\cp(G_i)+1$.

We assume without loss of generality that $z_1=C$. Note that from~(\ref{notwocp1}), $\cp(G_2^{AB})=\cp(G_2)$, hence $z_2\neq C$, and $\fvs(G_2^{AB})=\fvs(G_2)$ by~(\ref{fvscp}). We assume without loss of generality $z_2=A$. By symmetry, $\cp(G_1^{BC})=\cp(G_1)$ and $\fvs(G_1^{BC})=\fvs(G_1)$. From~(\ref{notthree0}) applied with $i=1$ and $x=A$, there is $y \in \{B,C\}$ such that $\fvs(G_2^{ABC-y})=\fvs(G_2)+1$. Note that $y \neq C$, so $y=B$ and $\fvs(G_2^{AC})=\fvs(G_2)+1$. We derive from~(\ref{fvscp}) that $\cp(G_2^{AC})=\cp(G_2)+1$, hence $\cp(G_1^{AC})=\cp(G_1)$ by~(\ref{notwocp1}). Again from~(\ref{fvscp}), we obtain $\fvs(G_1^{AC})=\fvs(G_1)$. 

Therefore we have: $\fvs(G_2^{AB})=\fvs(G_2)$, $\fvs(G_1^{BC})=\fvs(G_1)$, and $\fvs(G_1^{AC})=\fvs(G_1)$. 
Now, (\ref{notthree0}) applied with $i=2$ and $x=C$ yields a contradiction.
\cqed\end{proof}
\end{proof}




\section{Conclusion}

Through a non-trivial combination of elementary tricks and using a nice preliminary result of~\cite{munaro2016quelques}, we were able to close the case of Jones' Conjecture for subcubic graphs. 

The obvious question is whether this can be at all used to solve the whole conjecture. The reduction we have for subcubic graphs extends easily to the general setting, in the sense that a smallest counter-example to Jones' Conjecture is essentially $4$-edge-connected. It is not difficult to argue in a similar way that such a graph is $3$-vertex-connected. However, a much harder question is whether it is essentially $4$-vertex-connected. While it still seems possible, such a result using our approach would require additional tricks. Note that being in the general setting also gives us more leeway regarding possible reductions (no need to shy away from increasing the maximum degree, as long as there are fewer vertices). 

A second obstacle to generalization is that even assuming that a smallest counter-example is essentially $4$-vertex-connected, Theorem~\ref{munaro} only deals with the subcubic case. Another argument must then be devised.

A different approach would be not to aim for the conjectured bound of $2$ but simply for any bound better than the existing one of $3$. Unfortunately, this does not seem conceptually much easier. Let us emphasize this: a simple discharging argument yields $\fvs(G) \leq 3 \cp(G)$ for every planar graph $G$, while even significant effort fails to grant a factor of $(3-\epsilon)$ instead of $3$. 

To highlight how little we understand around Jones' Conjecture, we conclude by posing the following stronger conjecture. Note that the example of many nested disjoint cycles shows that the embedding cannot be fixed. Also note that the simple discharging argument mentioned above does not imply the following conjecture with a factor of $3$ instead of $2$.

\begin{Conjecture}\label{conj:fp}
For every planar graph $G$, we have $$\fvs(G)\leq 2 \cdot \fp(G),$$ where $\fp(G)$ is the maximum size of a face-packing of $G$, 
i.e., a cycle-packing where, for some embedding of $G$, every cycle bounds a face.
\end{Conjecture}

\bibliography{jones}

\begin{thebibliography}{CvBHJR92}

\bibitem[CFS12]{chen2012feedback}
Hong-Bin Chen, Hung-Lin Fu, and Chie-Huai Shih.
\newblock Feedback vertex set on planar graphs.
\newblock {\em Taiwanese Journal of Mathematics}, 16(6):2077--2082, 2012.

\bibitem[CGH14]{chappell2014cycle}
Glenn~G. Chappell, John Gimbel, and Chris Hartman.
\newblock On cycle packings and feedback vertex sets.
\newblock {\em Contributions to Discrete Mathematics}, 9(2), 2014.

\bibitem[CvBHJR92]{van2019tight}
Wouter Cames~van Batenburg, Tony Huynh, Gwena{\"e}l Joret, and Jean-Florent
  Raymond.
\newblock A tight {E}rdős-{P}ósa function for planar minors.
\newblock {\em Advances in Combinatorics}, 33pp, 2019:2.

\bibitem[EP65]{erdos_posa_1965}
Paul Erd\H{o}s and Lajos P\'osa.
\newblock On independent circuits contained in a graph.
\newblock {\em Canadian Journal of Mathematics}, 17:347--352, 1965.

\bibitem[KLL02]{kloks2002new}
Ton Kloks, Chuan-Min Lee, and Jiping Liu.
\newblock New algorithms for k-face cover, k-feedback vertex set, and
  k-disjoint cycles on plane and planar graphs.
\newblock In {\em International Workshop on Graph-Theoretic Concepts in
  Computer Science}, pages 282--295. Springer, 2002.

\bibitem[Mun16]{munaro2016quelques}
Andrea Munaro.
\newblock {\em Sur quelques invariants classiques et nouveaux des
  hypergraphes}.
\newblock PhD thesis, Grenoble Alpes, 2016.

\bibitem[MYZ13]{ma2013approximate}
Jie Ma, Xingxing Yu, and Wenan Zang.
\newblock Approximate min-max relations on plane graphs.
\newblock {\em Journal of Combinatorial Optimization}, 26(1):127--134, 2013.

\bibitem[RS86]{robertson1986graph}
Neil Robertson and Paul~D. Seymour.
\newblock Graph minors. {V}. excluding a planar graph.
\newblock {\em Journal of Combinatorial Theory, Series B}, 41(1):92--114, 1986.

\end{thebibliography}

\end{document}